\documentclass[a4paper]{amsart}
\usepackage[active]{srcltx}
\usepackage[all]{xy}
\usepackage{subfig}
\usepackage{hyperref}
\usepackage{mathrsfs}

\usepackage{amssymb,latexsym}
\usepackage{mathrsfs}
\usepackage{graphics}
\usepackage{latexsym}
\usepackage{psfrag}
\usepackage{verbatim}
\usepackage{bbold}
\usepackage[usenames]{color}
\usepackage{pifont,marvosym}

\theoremstyle{plain}
\newtheorem{lemma}{Lemma}[section]
\newtheorem{theorem}[lemma]{Theorem}
\newtheorem{proposition}[lemma]{Proposition}
\newtheorem{corollary}[lemma]{Corollary}

\theoremstyle{definition}

\newtheorem{definition}[lemma]{Definition}
\newtheorem{remark}[lemma]{Remark}

\numberwithin{equation}{section}

\newcommand{\supp}{\text{\rm supp}}
\newcommand{\loc}{\text{\rm loc}}

\newcommand{\ve}{\varepsilon}

\newcommand{\erre}{\mathbb{R}}
\newcommand{\enne}{\mathbb{N}}

\newcommand{\f}{\varphi}

\newcommand{\G}{\mathcal{G}}
\newcommand{\gammA}{\boldsymbol{\gamma}}

\renewcommand{\r}{\varrho}

\begin{document}
\title{Local Curvature-Dimension condition implies Measure-Contraction property}

\author{Fabio Cavalletti} 
\address{Hausdorff Center for Mathematics, Endenicher Allee  62, D-53115 Bonn}
\email{fabio.cavalletti@hcm.uni-bonn.de}

\author{Karl-Theodor Sturm}
\address{Institut f\"ur Angewandte Mathematik 
Abt. Stochastische Analysis,
Endenicher Allee 60, D-53115 Bonn}
\email{sturm@uni-bonn.de}

\begin{abstract}
We prove that for non-branching metric measure spaces the local curvature condition $\mathsf{CD}_{loc}(K,N)$ implies 
the global version of $\mathsf{MCP}(K,N)$. 
The curvature condition $\mathsf{CD}(K,N)$ introduced by the second author and also studied by Lott \& Villani is the 
generalization to metric measure space of lower bounds on Ricci curvature together with upper bounds on the dimension.
This paper is the following step of \cite{sturm:loc} where it is shown that $\mathsf{CD}_{loc}(K,N)$ is equivalent to 
a global condition $\mathsf{CD}^{*}(K,N)$, slightly weaker than the usual $\mathsf{CD}(K,N)$. 
It is worth pointing out that our result implies sharp Bishop-Gromov volume growth inequality and sharp Poincar\'e inequality.
\end{abstract}

\maketitle


\section{Introduction}

An important class of singular spaces is the one of
metric measure spaces with generalized lower bounds on the Ricci curvature formulated in terms of optimal transportation. 
This class of spaces together with the condition on lower bounds on curvature 
have been introduced by the second author in \cite{sturm:MGH1, sturm:MGH2} and 
independently by Lott and Villani in \cite{villott:curv}.

The condition called \emph{curvature-dimension condition} $\mathsf{CD}(K,N)$ depends on two parameters $K$ and $N$, playing the role
of a curvature and dimension bound, respectively. We recall two important properties of the condition $\mathsf{CD}(K,N)$:  
\begin{itemize}
\item the curvature-dimension condition is stable under convergence of metric measure spaces with respect to the $L^{2}$-transportation distance 
$\mathbb{D}$ introduced in \cite{sturm:MGH1}; 
\item a complete Riemannian manifold satisfies $\mathsf{CD}(K,N)$ if and only if its Ricci curvature is bounded from below by $K$ and 
its dimension from above by $N$.
\end{itemize}
Moreover a broad variety of geometric and functional analytic properties can be deduced from the curvature-dimension condition $\mathsf{CD}(K,N)$: 
the Brunn-Mikowski inequality, the Bishop-Gromov volume comparison theorem, the Bonnet-Myers theorem, the doubling property and local Poincar\'e inequalities on balls. All this listed results are quantitative results (volume of intermediate points, volume growth, upper bound 
on the diameter and so on) depending on $K,N$.

A variant of $\mathsf{CD}(K,N)$ is the \emph{measure-contraction property}, $\mathsf{MCP}(K,N)$, introduced in 
\cite{ohta:mcp} and \cite{sturm:MGH2}. In the setting of non-branching metric measure spaces it is proven that condition $\mathsf{CD}(K,N)$
implies $\mathsf{MCP}(K,N)$. While $\mathsf{CD}(K,N)$ is a condition on the optimal transport between  
any pair of absolutely continuous (w.r.t. $m$) probability measure on $M$, $\mathsf{MCP}(K,N)$ is a condition on the optimal transport
between Dirac masses and the uniform distribution $m$ on $M$.
Nevertheless a great part of the geometric and functional analytic properties verified by spaces satisfying the condition $\mathsf{CD}(K,N)$
are also verified by spaces satisfying the $\mathsf{MCP}(K,N)$:
\begin{itemize}
\item generalized Bishop-Gromov volume growth inequality;
\item doubling property;
\item a bound on the Hausdorff dimension; 
\item generalized Bonnet-Myers theorem.
\end{itemize}
Again this results are in a quantitative form depending on $K,N$.
For a complete list of analytic consequences of the measure contraction property see \cite{sturm:MGH2}.

Among the relevant questions on $\mathsf{CD}(K,N)$ that are still open, we are interested in studying the following one: 
can we say that a metric measure space $(M,d,m)$ satisfies $\mathsf{CD}(K,N)$ provided $\mathsf{CD}(K,N)$ 
holds true locally on a family of sets $M_{i}$ covering $M$? \\
In other words it is still not known whether $\mathsf{CD}(K,N)$ verifies the globalization property (or the local-to-global property).

A partial answer to this problem is contained in the work by Bacher and the second author \cite{sturm:loc}: they proved that 
if a metric measure space $(M,d,m)$ verifies
the local curvature-dimension condition $\mathsf{CD}_{loc}(K,N)$ then it verifies the global reduced curvature-dimension condition 
$\mathsf{CD}^{*}(K,N)$. The latter is strictly weaker than $\mathsf{CD}(K,N)$ and a converse implication can be obtained 
only changing the value of the lower bound on the curvature: condition $\mathsf{CD}^{*}(K,N)$ implies $\mathsf{CD}(K^{*},N)$ where 
$K^{*}= K(N-1)/N$. Therefore $\mathsf{CD}^{*}(K,N)$ gives worse geometric and analytic information than $\mathsf{CD}(K,N)$.

In this paper we prove that if $(M,d,m)$ is a non-branching metric measure space that verifies $\mathsf{CD}_{loc}(K,N)$ then $(M,d,m)$ verifies $\mathsf{MCP}(K,N)$.

Hence our result implies that from the local condition $\mathsf{CD}_{loc}(K,N)$ 
one can obtain all the global geometric and functional analytic consequences implied by $\mathsf{MCP}(K,N)$ 
and therefore the geometric and functional analytic consequences are obtained in the sharp quantitative version. 

We now present our approach to the problem.\\
As already pointed out, the curvature-dimension condition $\mathsf{CD}(K,N)$ prescribes  how the volume of a given set is affected by curvature 
when it is moved via optimal transportation. Condition $\mathsf{CD}(K,N)$ impose that the distortion is ruled by the coefficient $\tau_{K,N}^{(t)}(\theta)$
depending on the curvature $K$, on the dimension $N$, on the time of the evolution $t$ and on the point $\theta$. 

The main feature of the coefficient $\tau_{K,N}^{(t)}(\theta)$ is that it is obtained mixing two different information on how the volume should evolve:
an $(N-1)$-dimensional distortion depending on the curvature $K$ by  and a one dimensional evolution that doesn't feel the curvature. To be more precise
\[
\tau_{K,N}^{(t)}(\theta) = t^{1/N} \sigma_{K,N-1}^{(t)}(\theta)^{(N-1)/N},
\] 
where $\sigma_{K,N-1}^{(t)}(\theta)^{(N-1)/N}$ contains the information on the $(N-1)$-dimensional volume distortion and 
the evolution in the remaining direction is ruled just by $t^{1/N}$. This is a clear similarity with the Riemannian case.

Our aim is, starting from $\mathsf{CD}_{loc}(K,N)$, to isolate a local $(N-1)$-dimensional condition ruled by the coefficient $\sigma_{K,N-1}^{(t)}(\theta)$ 
and then, using the easier structure of $\sigma_{K,N-1}^{(t)}(\theta)$, obtain a global $(N-1)$-dimensional condition 
with coefficient $\sigma_{K,N-1}^{(t)}(\theta)$. At that point, using H\"older inequality and the linear behavior of the other direction, 
it is not difficult to pass from the $(N-1)$-dimensional version to the full-dimensional version with coefficient $\tau_{K,N}^{(t)}(\theta)$.

However to detect a local $(N-1)$-dimensional condition it is necessary to decompose the whole evolution into a family of $(N-1)$-dimensional evolutions.
Considering the optimal transport between a Dirac mass in $o$ and the uniform distribution $m$, the family of spheres around $x_{0}$
immediately provides the correct $(N-1)$-dimensional evolutions. This motivates why we obtain $\mathsf{MCP}(K,N)$ and not $\mathsf{CD}(K,N)$.

We state the main result of this paper.
\begin{theorem}[Theorem \ref{T:main}]
Let $(M,d,m)$ be a non-branching metric measure space. Assume that $(M,d,m)$ satisfies $\mathsf{CD}_{loc}(K,N)$. 
Then $(M,d,m)$ satisfies $\mathsf{MCP}(K,N)$.
\end{theorem}

We end this paper with an outlook on the most general case we can address using the approach described so far.

\section{Preliminaries}\label{S:preli}

Let $(M,d)$ be a metric space. 
The length $\mathsf{L}(\gamma)$ of a continuous curve $\gamma : [0,1] \to M$ is defined as
\[
\mathsf{L}(\gamma) : = \sup \sum_{k=1}^{n} d(\gamma(t_{k-1}),\gamma(t_{k}))
\]
where the supremum runs over $n \in \enne$ and over all partitions $0 = t_{0} < t_{1}< \dots < t_{n}=1$. 
Clearly $\mathsf{L}(\gamma) \geq d(\gamma(0),\gamma(1))$. The curve is called \emph{geodesic} if and only if 
$\mathsf{L}(\gamma) = d(\gamma(0),\gamma(1))$.
In this case we always assume that $\gamma$ has constant speed, i.e. 
$\mathsf{L}(\gamma\llcorner_{[s,t]}) =|s-t|\mathsf{L}(\gamma)= |s-t|d(\gamma(0),\gamma(1))$ for 
every $0\leq s \leq t \leq 1$.

With $\G(M)$ we denote the space of geodesic $\gamma: [0,1] \to M$ in $M$, regarded as subset of $Lip([0,1],M)$
of Lipschitz functions equipped with the topology of uniform convergence.

$(M,d)$ is said a \emph{length space} if and only if for all $x,y \in M$,
\[
d(x,y) = \inf \mathsf{L}(\gamma)
\]
where the infimum runs over all continuous curves connecting $x$ to $y$. It is said to be a \emph{geodesic space} if and only if 
every $x,y \in M$ are connected by a geodesic.
\begin{definition}
A geodesic space $(M,d)$ is \emph{non-branching} if and only if for all $r \geq 0$ and $x,y \in M$ such that $d(x,y)=r/2$ the set 
\[
\{z \in M : d (x,z) = r \} \cap \{ z \in M : d(y,z) = r/2 \}
\]
is a singleton.
\end{definition}

A \emph{metric measure space} will always be a triple $(M,d,m)$ where 
$(M,d)$ is a complete separable metric space and $m$ is a locally finite measure (i.e. $m(B_{r}(x))< \infty$ for all $x\in M$ 
and all sufficiently small $r>$0) on $M$ equipped with its Borel $\sigma$-algebra. We exclude the case $m(M)=0$.
A \emph{non-branching} metric measure space will be a metric measure space $(M,d,m)$ such that $(M,d)$ is a non-branching geodesic space.
Throughout the following we will use the notation $B_{p}(z) = \{ x : d(x,z)< p \}$.

\subsection{Geometry of metric measure spaces}\label{Ss:geom}

$\mathcal{P}_{2}(M,d)$ denotes the $L^{2}$-Wasserstein space of probability measures on $M$ and $d_{W}$ the corresponding $L^{2}$-Wasserstein distance. 
The subspace of $m$-absolutely continuous measures is denoted by $P_{2}(M, d, m)$.
A point $z$ will be called $t$-intermediate point of points $x$ and $y$ if $d(x, z)=td(x, y)$ and $d(z, y)=(1-t)d(x, y)$.
 
The following are well-known results in optimal transportation and are valid for general metric measure spaces.
\begin{lemma}\label{L:geod}
Let $(M,d,m)$ be a metric measure space.
For each geodesic $\Gamma: [0,1] \to \mathcal{P}_{2}(M)$ there exists a probability measure $\Xi$ on $\G(M)$ such that 
\begin{itemize}
\item $e_{t\,\sharp} \Xi = \Gamma(t)$ for all $t \in [0,1]$;
\item for each pair $(s,t)$ the transference plan $(\gamma_{s},\gamma_{t})_{\sharp} \Xi$ is an optimal coupling.
\end{itemize}
\end{lemma}

The curvature-dimension condition $\mathsf{CD}(K,N)$ is defined in terms of convexity properties of 
the lower semi-continuous R\'enyi entropy functional 
\begin{equation}\label{E:entropy}
\mathcal{S}_{N}(\mu | m) : = - \int_{M} \r^{-1/N}(x) \mu(dx)
\end{equation}
on $P_{2}(M,d)$ where $\r$ denotes the density of the absolutely continuous part $\mu^{c}$ in the Lebesgue decomposition 
$\mu = \mu^{c} + \mu^{s} = \r m + \mu^{s}$.

Given two numbers $K,N\in \erre$ with $N\geq1$, we put for $(t,\theta) \in[0,1] \times \erre_{+}$,
\begin{equation}\label{E:tau}
\tau_{K,N}^{(t)}(\theta):= 
\begin{cases}
\infty, & \textrm{if}\ K\theta^{2} \geq (N-1)\pi^{2}, \crcr
\displaystyle  t^{1/N}\Bigg(\frac{\sin(t\theta\sqrt{K/(N-1)})}{\sin(\theta\sqrt{K/(N-1)})}\Bigg)^{1-1/N} & \textrm{if}\ K\theta^{2} \leq (N-1)\pi^{2}, \crcr
t & \textrm{if}\ K \theta^{2}<0\ \textrm{or}\\& \textrm{if}\ K \theta^{2}=0\ \textrm{and}\ N=1,  \crcr
\displaystyle  t^{1/N}\Bigg(\frac{\sinh(t\theta\sqrt{-K/(N-1)})}{\sinh(\theta\sqrt{-K/(N-1)})}\Bigg)^{1-1/N} & \textrm{if}\ K\theta^{2} \leq 0 \ \textrm{and}\ N>1.
\end{cases}
\end{equation}

That is, $\tau_{K,N}^{(t)}(\theta): = t^{1/N} \sigma_{K,N-1}^{(t)}(\theta)^{(N-1)/N}$ where
\[
\sigma_{K,N}^{(t)}(\theta) = \frac{\sin(t\theta\sqrt{K/N})}{\sin(\theta\sqrt{K/N})}, 
\]
if $0 < K\theta^{2}<N\pi^{2}$ and with appropriate interpretation otherwise. Moreover we put 
\[
\varsigma_{K,N}^{(t)}(\theta): = \tau_{K,N}^{(t)}(\theta)^{N}.
\]
The coefficients $\tau_{K,N}^{(t)}(\theta),\sigma_{K,N}^{(t)}(\theta)$ and $\varsigma_{K,N}^{(t)}(\theta)$ are all volume distortion coefficients 
depending on the curvature $K$ and on the dimension $N$.

\begin{definition}[Curvature-Dimension condition]\label{D:CD}
Let two number $K,N \in \erre$ with $N\geq1$ be given. We say that $(M,d,m)$ satisfies the curvature-dimension condition 
- denoted  by $\mathsf{CD}(K,N)$ - if and only if for each pair 
$\nu_{0}, \nu_{1} \in \mathcal{P}_{2}(M,d,m)$ there exists an optimal coupling $\pi$ of $\nu_{0}=\r_{0}m$ and $\nu_{1}=\r_{1}m$,
and a geodesic $\Gamma:[0,1] \to \mathcal{P}_{2}(M,d,m)$ connecting $\nu_{0}$ and $\nu_{1}$ with 
\begin{equation}\label{E:CD}
\begin{aligned}
\mathcal{S}_{N'}(\Gamma(t)|m) \leq   - \int_{M\times M}& \Big[ \tau_{K,N'}^{(1-t)}(d(x_{0},x_{1}))\r_{0}^{-1/N'}(x_{0})  \crcr
&~ + \tau_{K,N'}^{(t)}(d(x_{0},x_{1}))\r_{1}^{-1/N'}(x_{1})  \Big]  \pi(dx_{0}dx_{1}),
\end{aligned}
\end{equation}
for all $t \in [0,1]$ and all $N'\geq N$.
\end{definition}

We recall also the definition of the reduced curvature-dimension condition $\mathsf{CD}^{*}(K,N)$ introduced in \cite{sturm:loc}
as well as the definition of $\mathsf{CD}_{loc}(K,N)$.

\begin{definition}[Reduced Curvature-Dimension condition]\label{D:CD*}
Let two number $K,N \in \erre$ with $N\geq1$ be given. We say that $(M,d,m)$ satisfies the reduced curvature-dimension condition 
- denoted  by $\mathsf{CD}^{*}(K,N)$ - if and only if for each pair 
$\nu_{0}, \nu_{1} \in \mathcal{P}_{2}(M,d,m)$ there exists an optimal coupling $\pi$ of $\nu_{0}=\r_{0}m$ and $\nu_{1}=\r_{1}m$,
and a geodesic $\Gamma:[0,1] \to \mathcal{P}_{2}(M,d,m)$ connecting $\nu_{0}$ and $\nu_{1}$ 
such that \eqref{E:CD} holds true for all $t \in [0,1]$ and all $N'\geq N$ 
with the coefficients $\tau_{K,N}^{(t)}(d(x_{0},x_{1}))$ and $\tau_{K,N}^{(1-t)}(d(x_{0},x_{1}))$
replaced by $\sigma_{K,N}^{(t)}(d(x_{0},x_{1}))$ and $\sigma_{K,N}^{(1-t)}(d(x_{0},x_{1}))$, respectively.
\end{definition}

\begin{definition}[Local Curvature-Dimension condition]\label{D:loc}
Let two number $K,N \in \erre$ with $N\geq1$ be given. We say that $(M,d,m)$ satisfies the curvature-dimension condition 
locally - denoted  by $\mathsf{CD}_{loc}(K,N)$ - if and only if each point $x \in M$ has a neighborhood $M(x)$ such that for each pair 
$\nu_{0}, \nu_{1} \in \mathcal{P}_{2}(M,d,m)$ supported in $M(x)$ there exists an optimal coupling $\pi$ of $\nu_{0}=\r_{0}m$ and $\nu_{1}=\r_{1}m$,
and a geodesic $\Gamma:[0,1] \to \mathcal{P}_{2}(M,d,m)$ connecting $\nu_{0}$ and $\nu_{1}$ with 
\begin{equation}\label{E:CDloc}
\begin{aligned}
\mathcal{S}_{N'}(\Gamma(t)|m) \leq   - \int_{M\times M}& \Big[ \tau_{K,N'}^{(1-t)}(d(x_{0},x_{1}))\r_{0}^{-1/N'}(x_{0})  \crcr
&~ + \tau_{K,N'}^{(t)}(d(x_{0},x_{1}))\r_{1}^{-1/N'}(x_{1})  \Big]  \pi(dx_{0}dx_{1}),
\end{aligned}
\end{equation}
for all $t \in [0,1]$ and all $N'\geq N$.
\end{definition}
Notice that the geodesic $\Gamma$ of the above definition can exit from the neighborhood $M(x)$.

As already emphasized in the introduction, in \cite{sturm:loc} it is proved that $\mathsf{CD}_{loc}(K,N)$ implies $\mathsf{CD}^{*}(K,N)$.

If a non-branching metric measure space $(M,d,m)$ satisfies $\mathsf{CD}(K,N)$ then the uniqueness of geodesics can be proven. The next result is taken 
from \cite{sturm:MGH2}.
\begin{lemma}\label{L:map}
Assume that $(M,d,m)$ is non-branching an satisfies $\mathsf{CD}(K,N)$ for some pair $(K,N)$. Then for every $x\in\supp[m]$ and $m$-a.e. $y\in M$ 
(with the exceptional set depending on x) there exists a unique geodesic between $x$ and $y$. 

Moreover, there exists a measurable map $\gamma: M^{2} \to \G(M)$ such that for $m\otimes m$-a.e. $(x,y) \in M^{2}$ the curve $t \mapsto \gamma_{t}(x,y)$
is the unique geodesic connecting $x$ and $y$.
\end{lemma}

In the setting of non-branching metric measure space $\mathsf{CD}(K,N)$ has an equivalent point-wise formulation:  $(M,d,m)$ satisfies 
$\mathsf{CD}(K,N)$ if and only if for each pair $\nu_{0},\nu_{1}\in \mathcal{P}_{2}(M,d,m)$ and each optimal coupling 
$\pi$ of them 
\begin{equation}\label{E:cdpunto}
\r_{t}(\gamma_{t}(x_{0},x_{1}))\leq \Big[ \tau_{K,N'}^{(1-t)}(d(x_{0},x_{1}))\r_{0}^{-1/N'}(x_{0}) + 
\tau_{K,N'}^{(t)}(d(x_{0},x_{1}))\r_{1}^{-1/N'}(x_{1})  \Big]^{-N}, 
\end{equation}
for all $t \in [0,1]$, and $\pi$-a.e. $(x_{0},x_{1}) \in M \times M$. Here $\r_{t}$ is the density of the push-forward of $\pi$ under 
the map $(x_{0},x_{1}) \mapsto \gamma_{t}(x_{0},x_{1})$.

We recall the definition of the measure contraction property. \\
A Markov kernel on $M$ is a map $Q: M\times \mathcal{B}(M) \to [0,1]$ 
(where $\mathcal{B}(M)$ denotes the Borel $\sigma$-algebra of M) with the following properties: 
\begin{itemize}
\item[(i)] for each $x \in M$ the map $Q(x, \cdot) : \mathcal{B}(M) \to [0,1]$ is a probability measure on $M$;
\item[(ii)] for each $A \in \mathcal{B}(M)$ the function $Q (\cdot, A) : M \to [0,1]$ is  $m$-measurable.
\end{itemize}

\begin{definition}[Measure contraction property]\label{D:mcp}
Let two number $K,N \in \erre$ with $N\geq1$ be given. We say that $(M,d,m)$ satisfies the \emph{measure contraction property} 
 $\mathsf{MCP}(K,N)$ if and only if  for each $0 < t < 1$ there exists a Markov kernel $Q_{t}$ from $M^{2}$ to $M$ 
 such that for $m^{2}$-a.e. $(x,y) \in M$ and for $Q_{t}(x,y; \cdot)$-a.e. $z$ the point $z$ is a $t$-intermediate point of $x$ and $y$,
 and such that for $m$-a.e. $x \in M$ and for every measurable $B \subset M$,
\begin{equation}\label{E:MCP}
\begin{aligned}
\int_{M} \varsigma^{(t)}_{K,N}(d(x,y))Q_{t}(x,y;B) m(dy)  \leq & ~ m(B), \crcr
\int_{M} \varsigma^{(1-t)}_{K,N}(d(x,y))Q_{t}(y,x;B) m(dy)  \leq & ~ m(B). 
\end{aligned}
\end{equation}
\end{definition}

\subsection{Disintegration of measures}
\label{Ss:disintegrazione}

Given a measurable space $(R, \mathscr{R})$ and a function $r: R \to S$, with $S$ generic set, we can endow $S$ with the \emph{push forward $\sigma$-algebra} $\mathscr{S}$ of $\mathscr{R}$:
\[
Q \in \mathscr{S} \quad \Longleftrightarrow \quad r^{-1}(Q) \in \mathscr{R},
\]
which could be also defined as the biggest $\sigma$-algebra on $S$ such that $r$ is measurable. Moreover given a measure space 
$(R,\mathscr{R},\rho)$, the \emph{push forward measure} $\eta$ is then defined as $\eta := (r_{\sharp}\rho)$.

Consider a probability space $(R, \mathscr{R},\rho)$ and its push forward measure space $(S,\mathscr{S},\eta)$ induced by a map $r$. From the above definition the map $r$ is clearly measurable and inverse measure preserving.

\begin{definition}
\label{defi:dis}
A \emph{disintegration} of $\rho$ \emph{consistent with} $r$ is a map $\rho: \mathscr{R} \times S \to [0,1]$ such that
\begin{enumerate}
\item  $\rho_{s}(\cdot)$ is a probability measure on $(R,\mathscr{R})$ for all $s\in S$,
\item  $\rho_{\cdot}(B)$ is $\eta$-measurable for all $B \in \mathscr{R}$,
\end{enumerate}
and satisfies for all $B \in \mathscr{R}, C \in \mathscr{S}$ the consistency condition
\[
\rho\left(B \cap r^{-1}(C) \right) = \int_{C} \rho_{s}(B) \eta(ds).
\]
A disintegration is \emph{strongly consistent with respect to $r$} if for all $s$ we have $\rho_{s}(r^{-1}(s))=1$.
\end{definition}

The measures $\rho_s$ are called \emph{conditional probabilities}.

We say that a $\sigma$-algebra $\mathcal{H}$ is \emph{essentially countably generated} with respect to a measure $m$ if there exists a countably generated $\sigma$-algebra $\hat{\mathcal{H}}$ such that for all $A \in \mathcal{H}$ there exists $\hat{A} \in \hat{\mathcal{H}}$ such that $m (A \vartriangle \hat{A})=0$.

We recall the following version of the disintegration theorem. See \cite{biacar:cmono} for a direct proof.

\begin{theorem}[Disintegration of measures]
\label{T:disintr}
Assume that $(R,\mathscr{R},\rho)$ is a countably generated probability space, $R = \cup_{s \in S}R_{s}$ a partition of R, $r: R \to S$ the quotient map and $\left( S, \mathscr{S},\eta \right)$ the quotient measure space. Then  $\mathscr{S}$ is essentially countably generated w.r.t. $\eta$ and there exists a unique disintegration $s \mapsto \rho_{s}$ in the following sense: if $\rho_{1}, \rho_{2}$ are two consistent disintegration then $\rho_{1,s}(\cdot)=\rho_{2,s}(\cdot)$ for $\eta$-a.e. $s$.

If $\left\{ S_{n}\right\}_{n\in \enne}$ is a family essentially generating  $\mathscr{S}$ define the equivalence relation:
\[
s \sim s' \iff \   \{  s \in S_{n} \iff s'\in S_{n}, \ \forall\, n \in \enne\}.
\]
Denoting with p the quotient map associated to the above equivalence relation and with $(L,\mathscr{L}, \lambda)$ the quotient measure space, the following properties hold:
\begin{itemize}
\item $R_{l}:= \cup_{s\in p^{-1}(l)}R_{s} = (p \circ r)^{-1}(l)$ is $\rho$-measurable and $R = \cup_{l\in L}R_{l}$;
\item the disintegration $\rho = \int_{L}\rho_{l} \lambda(dl)$ satisfies $\rho_{l}(R_{l})=1$, for $\lambda$-a.e. $l$. In particular there exists a 
strongly consistent disintegration w.r.t. $p \circ r$;
\item the disintegration $\rho = \int_{S}\rho_{s} \eta(ds)$ satisfies $\rho_{s}= \rho_{p(s)}$ for $\eta$-a.e. $s$.
\end{itemize}
\end{theorem}

In particular we will use the following corollary.

\begin{corollary}
\label{C:disintegration}
If $(S,\mathscr{S})=(X,\mathcal{B}(X))$ with $X$ Polish space, then the disintegration is strongly consistent.
\end{corollary}

\section{Polar coordinates}\label{S:polar}
From now on we will assume $(M,d,m)$ to be a non-branching metric measure space satisfying $\mathsf{CD}_{loc}(K,N)$ for some $K,N \in \erre$ and 
$N \geq 1$. Since we want to prove that $(M,d,m)$ satisfies $\mathsf{MCP}(K,N)$ we also fix once forever $o \in M$.  

Decompose $M = \cup_{r\geq 0}M_{r}$ with $M_{r} : = \partial B_{r}(o)$ and, accordingly to this decomposition, 
$m$ can be disintegrated in the following way
\[
m = \int \bar m_{r} q(dr), \qquad   q(A) = m (\{ x : d(x,0) \in A\}).
\]
It is fairly easy to prove that the disintegration is strongly consistent. 
Indeed restrict $m$ to $B_{R}(o)$, with $R>0$, and consider any constant speed geodesic $\gamma$ going from $o$ to $M_{R}$
and take $[0, R]$ as the quotient set. 
It follows that the quotient space is a Polish space and then by Corollary \ref{C:disintegration} the disintegration is strongly consistent.
Letting $R \nearrow +\infty$, we obtain the strong consistency for the whole measure and 
$q$ will be a locally finite measure, therefore:

\[
\bar m_{r}(\{ x : d(x,o)= r\}) =1, \qquad \textrm{for }  q-a.e.\ r \in [0,\bar R].
\]

\begin{proposition}\label{P:retta}
The quotient measure $q \ll \mathcal{L}^{1} $.
\end{proposition}
\begin{proof}
Since $(M,d,m)$ satisfies $\mathsf{CD}_{\loc}(K,N)$, from \cite{sturm:loc} $(M,d,m)$ verifies $\mathsf{CD}^{*}(K,N)$, then defining 
\[
v(r):= m(\bar B_{r} (o), \qquad s(r) : = \limsup_{\delta \to 0} \frac{1}{\delta}  m (\bar B_{r+\delta}(o)\setminus B_{r}(o) ),
\]
the map $r \mapsto  v(r)$ is locally Lipschitz with $s$ as weak derivative, Theorem 2.3. of \cite{sturm:MGH2}. 
Being $s$ the density of $q$ w.r.t. $\mathcal{L}^{1}$, it follows that $q\ll \mathcal{L}^{1}$.
\end{proof}

With a slight abuse of notation $q (dr) = q(r)  \mathcal{L}^{1}$. Let $m_{r} : = q(r) \bar m_{r}$
so we have 
\[
m = \int  m_{r} dr. 
\]

Let $s_{r}:=m_{r}(M) = m_{r}(M_{r}) = \frac{d^{+}}{dr} m(B_{r}(o))$ and note that reduced Bishop-Gromov inequality, see \cite{sturm:loc}, 
implies that for all $0< r \leq R\leq \pi \sqrt{(N-1)/K^{*}}$
\begin{equation}\label{E:surfacevo}
\frac{s_{r}}{s_{R}} \geq\bigg( \frac{\sin(r\sqrt{K^{*}/(N-1)})}{\sin(R\sqrt{K^{*}/(N-1)}} \bigg)^{N},
\end{equation}
where $K^{*} = K(N-1)/N$.

Fix $R>0$ with $s_{R}>0$ and let $(p_{r})_{r\in [0,R]}$ denote the 
geodesic in $\mathcal{P}(M)$ connecting the probability measures $p_{0}  = \delta_{x_{0}}$ and 
$p_{R} = \frac{1}{s_{R}}m_{R}$. Note that for each $r$ the measure $p_{r}$ is supported on $M_{r}$. 
The next lemma follows straightforwardly from \eqref{E:surfacevo}.

\begin{lemma}\label{L:surfacevo}
The measure $p_{r}$ is absolute continuous with respect to the surface measure $m_{r}$.
\end{lemma}
Let $\hat h_{r} (x) : =  \frac{d p_{r}} {dm_{r}}(x)$ denote the density. Clearly $\hat h_{r}$ can be defined arbitrarily outside $M_{r}$.
Therefore for $\mathcal{L}^{1}$-a.e.  $p_{r} = \hat h_{r} m_{r}$.

\begin{remark} \label{R:lifting}
Let us consider the set of geodesic  
\[
\G_{[0,R]}(M):= \{ \gamma : [0, R] \to M, \textrm{constant speed geodesic} \}.
\]
Let $\nu \in \mathcal{P}(\G_{[0,R]}(M))$ such that for $\mathcal{L}^{1}$-a.e. $r \in [0,R]$, $e_{r\,\sharp} \nu = p_{r}$.
Neglecting a set of arbitrarily small $\nu$-measure, we assume w.l.o.g. that 
\[
G: = \supp[\nu] \subset \G_{[0,R]}(M),\quad 
\hat G_{r}: = e_{r}(G) \subset M_{r},  \quad \hat G : = \cup_{r\in [0,R]}\hat G_{r} \subset M,
\]
with $G$ compact and the maps $e_{r} : G \to \hat G_{r}$ and 
\[
\begin{aligned}
e: (0,R) \times G &~ \to  \hat G \crcr
(r,\gamma) &~\mapsto  e_{r}(\gamma) : = \gamma_{r}
\end{aligned}
\]
are both homeomorphisms. We also prefer to think of $\hat h_{r}$ as a function defined on $G$ rather than on $\hat G$, hence define 
$h_{r}: G \to [0,\infty]$ by $h_{r}(\gamma):= \hat h_{r}(\gamma_{r})$.
\end{remark}

\section{The $(N-1)$-dimensional estimate}\label{S:reduction}

Consider $H \subset G$, $\nu$-measurable with $\nu(H)>0$ and numbers $R_{0},L_{0},R_{1},L_{1}>0$ with $R_{0} < R_{1}$
such that $R_{t}+L_{t}<R$ for all $t \in [0,1]$ where $R_{t}:= (1-t) R_{0} + tR_{1}$ and $L_{t}:= (1-t)L_{0} + tL_{1}$, then the following holds. 

\begin{lemma}
The curve
\begin{equation}\label{E:curve}
t \mapsto \mu_{t}  : =  \frac{1}{L_{t}\nu(H)} \int_{0}^{R} 1_{(R_{t},R_{t}+L_{t})\times H}(e^{-1}(x))   p_{r}(dx) \mathcal{L}^{1}(dr) \in \mathcal{P}(M)
\end{equation}
is a geodesic.
\end{lemma}

\begin{proof}
Observe that coupling each $\gamma_{R_{s}+\lambda L_{s}}$ with $\gamma_{R_{t}+\lambda L_{t}}$ for $\lambda \in [0,1], \gamma \in H$
we obtain a $d^{2}$-cyclically monotone coupling of $\mu_{s}$ with $\mu_{t}$. The property then follows straightforwardly.
\end{proof}

Hence, the optimal transport is achieved by not changing the ``angular'' parts and coupling radial parts according to 
optimal coupling on $\erre$. Observe that for each $t \in [0,1]$ the density $\r_{t}(x)$ of $\mu_{t}$ w.r.t. $m$ is given by
\begin{equation}\label{E:density}
\r_{t}(\gamma_{r}) =
\begin{cases}
\displaystyle \frac{1}{L_{t}\nu(H)}  h_{r}(\gamma), & (r,\gamma) \in  [R_{t},R_{t} + L_{t}] \times H, \crcr 
0, & \textrm{otherwise}.
\end{cases}
\end{equation}

The following regularity result for densities holds true. 
\begin{lemma}\label{L:regular}
For $\nu$-a.e. $\gamma \in G$, the function $r \mapsto h_{r}^{-1/N}(\gamma)$ is semi-concave on $(0,R)$ and satisfies in distributional sense 
\[
\partial^{2}_{r} h_{r}^{-1/N}(\gamma) \leq -\frac{K}{N} h_{r}^{-1/N}(\gamma).
\]
\end{lemma}
\begin{proof}
Recall that $\mathsf{CD}_{loc}(K,N)$ implies $\mathsf{CD}^{*}(K,N)$.
Consider the geodesic $\mu_{t}$ defined in \eqref{E:curve} with $L_{0} = L_{1} = 1$ and apply the definition of $\mathsf{CD}^{*}(K,N)$
to get 
\begin{equation}\label{E:semi-concave}
h_{s}^{-1/N}(\gamma)\geq \frac{\sin (t-s)\sqrt{K/N}}{\sin (t-r)\sqrt{K/N}}h_{r}^{-1/N}(\gamma) +
\frac{\sin (s-r)\sqrt{K/N}}{\sin (t-r)\sqrt{K/N}}h_{t}^{-1/N}(\gamma),
\end{equation}
for all $0<r<s<t<R$ and $\nu$-a.e. $\gamma \in G$. The claim is equivalent to \eqref{E:semi-concave}.
\end{proof}

Now fix an open set $H \subset G$ and $[a,b] \subset [0,R]$ such that the curvature dimension condition $\mathsf{CD}(K,N)$ holds
true for all measures $\mu_{0},\mu_{1}$ supported in $e([a,b]\times \bar H)$.
For each $R_{0},R_{1} \in (a,b)$ choose $L_{0},L_{1}$ such that $R_{0}+L_{0}, R_{1}+L_{1} \leq b$ and define $(\mu_{t})_{t \in [0,1]}$ as before in 
\eqref{E:curve}.
Moreover we have to consider the following map $\Phi : \G_{[0,R]}(M) \times [0,1] \to \G_{[0,1]}(M)$ with $\Phi(\gamma,s)$ being the geodesic
$t \mapsto \eta_{t} = \gamma_{(1-t)(R_{0} +s L_{0}) + t (R_{1}+s L_{1})}$. Consider  
\[
\tilde \nu : = \Phi_{\sharp} \bigg( \frac{1}{\nu(H)} \nu\llcorner_{H} \otimes \mathcal{L}^{1}\llcorner_{[0,1]} \bigg),
\]
then $\mu_{t} = e_{t\,\sharp} \tilde \nu $.

\begin{theorem}\label{T:surface}
For $\nu$-a.e. $\gamma \in H$ and for sufficiently close $R_{0}<R_{1}$ the following holds true:
\begin{equation}\label{E:surface}
h_{R_{1/2}}^{-\frac{1}{N-1}}(\gamma)  
 \geq  \sigma_{K,N-1}^{(1/2)}(R_{1}-R_{0}) \left\{   h_{R_{0}}^{-\frac{1}{N-1}}(\gamma) +   h_{R_{1}}^{-\frac{1}{N-1}}(\gamma) \right\}.
\end{equation}
\end{theorem}

\begin{proof}
Consider the measures $\mu_{0}$ and $\mu_{1}$, the corresponding measure on the space of geodesics $\tilde \nu$
and recall that  $\mu_{t} = \r_{t} m$.

{\it Step 1.}
Condition $\mathsf{CD}_{loc}(K,N)$ for $t=1/2$ and the assumptions on $R_{0},L_{0}$ and $R_{1},L_{1}$ imply that for $\tilde \nu$-a.e. 
$\eta \in \G_{[0,1]}(M)$
\[
\r^{-1/N}_{1/2}(\eta_{1/2}) \geq \tau_{K,N}^{(1/2)}(d(\eta_{0},\eta_{1})) \left\{ \r_{0}^{-1/N}(\eta_{0}) + \r_{1}^{-1/N}(\eta_{1}) \right\},
\]
that can be formulated also in the following way: for $\mathcal{L}^{1}$-a.e. $s\in [0,1]$ and $\nu$-a.e. $\gamma \in H$ 
\[
\r^{-1/N}_{1/2}(\gamma_{R_{1/2} + s L_{1/2}}) \geq 
\tau_{K,N}^{(1/2)}(R_{1}-R_{0} + s | L_{1} - L_{0}| ) \left\{ \r_{0}^{-1/N}(\gamma_{R_{0}+sL_{0}}) + \r_{1}^{-1/N}(\gamma_{R_{1}+s L_{1}}) \right\}.
\]
Then using \eqref{E:density} and the continuity of $r \mapsto h_{r}(\gamma)$ (Proposition \ref{L:regular}), letting $s \searrow 0$, it follows that
\begin{equation}\label{E:splitting}
(L_{0}+L_{1})^{1/N} h^{-1/N}_{R_{1/2}}(\gamma) \geq 
\sigma_{K,N}^{(1/2)}(R_{1}-R_{0})^{\frac{N-1}{N}} \left\{L_{0}^{1/N} h_{R_{0}}^{-1/N}(\gamma) +L_{1}^{1/N} h_{R_{1}}^{-1/N}(\gamma) \right\}
\end{equation}
for all $R_{0} < R_{1} \in (a,b)$, all sufficiently small $L_{0},L_{1}$ and $\nu$-a.e. $\gamma \in H$, with exceptional set depending on 
$R_{0},R_{1},L_{0},L_{1}$.

{\it Step 2.} Note that all the involved quantities in \eqref{E:splitting} are continuous w.r.t. $R_{0},R_{1},L_{0},L_{1}$, therefore there exists a common 
exceptional set $H' \subset H$of zero $\nu$-measure such that \eqref{E:splitting} holds true for all 
for all $R_{0} < R_{1} \in (a,b)$, all sufficiently small $L_{0},L_{1}$ and all $\gamma  \in H \setminus H'$.

For fixed $R_{0} < R_{1} \in (a,b)$ and fixed $\gamma \in H\setminus H'$, varying $L_{0},L_{1}$ in \eqref{E:splitting} yields 
\[
h_{R_{1/2}}^{-\frac{1}{N-1}}(\gamma)  
 \geq  \sigma_{K,N-1}^{(1/2)}(R_{1}-R_{0}) \left\{   h_{R_{0}}^{-\frac{1}{N-1}}(\gamma) +   h_{R_{1}}^{-\frac{1}{N-1}}(\gamma) \right\}.
\]
Indeed the optimal choice is 
\[
L_{0}  = L \frac{h_{R_{0}}^{-1/(N-1)}(\gamma)  }{h_{R_{0}}^{-1/(N-1)}(\gamma) + h_{R_{1}}^{-1/(N-1)}(\gamma)}, \qquad
L_{1}  = L \frac{h_{R_{1}}^{-1/(N-1)}(\gamma)  }{h_{R_{0}}^{-1/(N-1)}(\gamma) + h_{R_{1}}^{-1/(N-1)}(\gamma)}
\]
for sufficiently small $L>0$.
\end{proof}

\section{The global estimates}\label{S:global}
From Theorem \ref{T:surface} we have that for every fixed $\gamma \in G\setminus H'$: 
for every $0< R_{0} <  R$ there exists $\ve>0$ such that for all $R_{0} <R_{1}< R_{0}+\ve$ it holds
\[
\begin{aligned}
h_{R_{1/2}}^{-\frac{1}{N-1}}(\gamma)  
 \geq  \sigma_{K,N-1}^{(1/2)}(R_{1}-R_{0}) \left\{   h_{R_{0}}^{-\frac{1}{N-1}}(\gamma) +   h_{R_{1}}^{-\frac{1}{N-1}}(\gamma) \right\}.
\end{aligned}
\]
We prove that mid-points inequality is equivalent to the complete inequality.

\begin{lemma}[Midpoints]\label{L:mid}
Inequality \eqref{E:surface} holds true if and only if
\begin{equation}\label{E:reale}
h_{ R_{t}}^{-\frac{1}{N-1}} (\gamma)  \geq 
\sigma_{K,N-1}^{(1-t)}(R_{1}-R_{0})  h_{ R_{0}}^{-\frac{1}{N-1}}(\gamma) + 
\sigma_{K,N-1}^{(t)}(R_{1}-R_{0})h_{ R_{1}}^{-\frac{1}{N-1}} (\gamma)
\end{equation}
for all $t \in [0,1]$.
\end{lemma}

\begin{proof}
We only consider the case $K>0$. The general case requires analogous calculations.  
Fix $0 \leq R_{0} \leq   R_{1} \leq R$, put $\theta:=  R_{1} -  R_{0}$ and $h(s): = h_{s}(\gamma) = h(\gamma(s))$.

{\it Step 1.} For every $k \in \enne$ we have 
\[
\begin{aligned}
h^{-\frac{1}{N-1}}(R_{0} + l 2^{-k}\theta))   \geq &~
\sigma_{K,N-1}^{(1/2)}(2^{-k+1}\theta)  h^{-\frac{1}{N-1}}(R_{0} + (l -1)2^{-k}\theta))  \crcr
&~ +\sigma_{K,N-1}^{(1/2)}(2^{-k+1}\theta) h^{-\frac{1}{N-1}} (R_{0} + (l+1)2^{-k}\theta)),
\end{aligned}
\]
for every odd $ l = 0, \dots, 2^{k}$.

{\it Step 2.} We perform an induction argument on $k$: suppose that inequality \eqref{E:reale} is satisfied for all $t = l 2^{-k+1} \in [0,1]$ with $l$ odd, 
then $\eqref{E:reale}$ is verified by every $t = l2^{-k} \in [0,1]$ with $l$ odd:
\[
\begin{aligned}
h^{-\frac{1}{N-1}} (R_{0} +& l2^{-k}\theta))  \crcr
\geq &~ \sigma_{K,N-1}^{(1/2)}(2^{-k+1}\theta)  h^{-\frac{1}{N-1}} (R_{0} + (l-1)2^{-k}\theta)) \crcr
&~ +\sigma_{K,N-1}^{(1/2)}(2^{-k+1}\theta) h^{-\frac{1}{N-1}}(R_{0} + (l+1)2^{-k}\theta))  \crcr
\geq &~  \sigma_{K,N-1}^{(1/2)}(2^{-k+1}\theta) 
\Big[h^{-\frac{1}{N-1}}(R_{0}) \sigma_{K,N-1}^{(1-(l-1)2^{-k})}(\theta)+ h^{-\frac{1}{N-1}}(R_{1}) 
\sigma_{K,N-1}^{((l-1)2^{-k})}(\theta) \Big]  \crcr
&~ +  \sigma_{K,N-1}^{(1/2)}(2^{-k+1}\theta) 
\Big[h^{-\frac{1}{N-1}}(R_{0}) \sigma_{K,N-1}^{(1-(l+1)2^{-k})}(\theta)+ h^{-\frac{1}{N-1}}(R_{1}) 
\sigma_{K,N-1}^{((l+1)2^{-k})}(\theta) \Big].
\end{aligned}
\]
Following the calculation of the proof of Proposition 2.10 of \cite{sturm:loc}, one obtain that 
\[
\begin{aligned}
h^{-\frac{1}{N-1}}(R_{0} + l 2^{-k}\theta))   \geq
 \sigma_{K,N-1}^{(1 - l 2^{-k})}(\theta)   h^{-\frac{1}{N-1}}(R_{0} )
 + \sigma_{K,N-1}^{(l 2^{-k})}(\theta) h^{-\frac{1}{N-1}}(R_{1}).
\end{aligned}
\]
The claim is easily proved by the continuity of $h$ and $\sigma$.
\end{proof}

We prove that \eqref{E:reale} satisfies a local-to-global property.

\begin{theorem}[Local to Global]\label{T:loctoglob}
Suppose that for every $R_{0} \in [0,R]$ there exists $\ve>0$ such that whenever $R_{0}<R_{1}< R_{0}+\ve$ then \eqref{E:reale} holds true for all $t \in [0,1]$. Then \eqref{E:reale} holds true for all $0 < R_{0} < R_{1} \leq R$ and $t \in [0,1]$.
\end{theorem}

\begin{proof}
We only consider the case $K>0$. The general case requires analogous calculations.  
Fix $0 < R_{0} < R_{1} \leq R$, $\theta : =R_{1} - R_{0} $ and $h(s): = h_{s}(\gamma) = h(\gamma(s))$.

{\it Step 1.} According to our assumption, every point $R_{0} \in [0, R]$ has a neighborhood $(R_{0}-\ve(R_{0}),R+\ve(R_{0}))$ such that if $R_{1}$
belongs to that neighborhood then \eqref{E:reale} is verified. 
By compactness of $[0, R]$ there exist $x_{1},\dots, x_{n}$ such that the family $\{B_{\ve(x_{i})/2}(x_{i}) \}_{i=1,\dots,n}$ is a covering 
of $[0, R]$.Let $\lambda : = \min \{ \ve(x_{i})/2 : i =1,\dots,n \}$. 
Possibly taking a lower value for $\lambda$, we assume that $\lambda = 2^{-k}\theta$. 
Hence  we have 
\[
\begin{aligned}
h^{-\frac{1}{N-1}}(R_{0} + \frac{1}{2}\theta)   \geq &~
\sigma_{K,N-1}^{(1/2)}(2^{-k+1}\theta)h^{-\frac{1}{N-1}}(R_{0} + \frac{1}{2}\theta -2^{-k}\theta   )   \crcr
&~ +\sigma_{K,N-1}^{(1/2)}(2^{-k+1}\theta) h^{-\frac{1}{N-1}}(R_{0} + \frac{1}{2}\theta + 2^{-k}\theta).
\end{aligned}
\]

{\it Step 2.} We iterate the above inequality:
\[
\begin{aligned}
h^{-\frac{1}{N-1}}(R_{0} + \frac{1}{2}\theta)   \geq &~
\sigma_{K,N-1}^{(1/2)}(2^{-k+1}\theta)h^{-\frac{1}{N-1}}(R_{0} + \frac{1}{2}\theta -2^{-k}\theta   )   \crcr
&~ + \sigma_{K,N-1}^{(1/2)}(2^{-k+1}\theta) h^{-\frac{1}{N-1}}(R_{0} + \frac{1}{2}\theta + 2^{-k}\theta) \crcr
\geq &~  \sigma_{K,N-1}^{(1/2)}(2^{-k+1}\theta) 
\Big[\sigma_{K,N-1}^{(1/2)}(2^{-k+1}\theta) h^{-\frac{1}{N-1}}(R_{0} + \frac{1}{2}\theta -2^{-k+1}\theta )   \crcr
&~ + \sigma_{K,N-1}^{(1/2)}(2^{-k+1}\theta) h^{-\frac{1}{N-1}}(R_{0} + \frac{1}{2}\theta)   \Big]  \crcr
& ~+ \sigma_{K,N-1}^{(1/2)}(2^{-k+1}\theta) 
\Big[\sigma_{K,N-1}^{(1/2)}(2^{-k+1}\theta) h^{-\frac{1}{N-1}}(R_{0} + \frac{1}{2}\theta  )   \crcr
&~ + \sigma_{K,N-1}^{(1/2)}(2^{-k+1}\theta) h^{-\frac{1}{N-1}}(R_{0} + \frac{1}{2}\theta + 2^{-k+1}\theta)   \Big] \crcr
\geq &~  \sigma_{K,N-1}^{(1/2)}(2^{-k+1}\theta)^{2} 
h^{-\frac{1}{N-1}}(R_{0} + \frac{1}{2}\theta -2^{-k+1}\theta )  \crcr
&~ + \sigma_{K,N-1}^{(1/2)}(2^{-k+1}\theta)^{2} 
h^{-\frac{1}{N-1}}(R_{0} + \frac{1}{2}\theta + 2^{-k+1}\theta).
\end{aligned}
\]
Observing that $\sigma_{K,N-1}^{(1/2)}(\alpha)^{2}\geq \sigma_{K,N-1}^{(1/2)}(2\alpha)$, it is fairly easy to obtain: 
\[
\begin{aligned}
h^{-\frac{1}{N-1}}(R_{0} + \frac{1}{2}\theta)   \geq &~
\sigma_{K,N-1}^{(1/2)}(2^{-k+i+1}\theta)  h^{-\frac{1}{N-1}}(R_{0} + \frac{1}{2}\theta -2^{-k+i}\theta   )  \crcr
&~ +\sigma_{K,N-1}^{(1/2)}(2^{-k+i+1}\theta) h^{-\frac{1}{N-1}}(R_{0} + \frac{1}{2}\theta + 2^{-k+i}\theta) ,
\end{aligned}
\]
for every $i = 0,\dots, k$. For $i = k - 1$ Lemma \ref{L:mid} implies the claim.
\end{proof}

\section{From local $\mathsf{CD}(K,N)$ to $\mathsf{MCP}(K,N)$}\label{S:globmcp}

So we have proved that for any $0<R_{0} < R_{1} < R$ the density $h_{r}$, of $p_{r}$ w.r.t. $m_{r}$, 
satisfies the following inequality:
\begin{equation}\label{E:final}
\begin{aligned}
h^{-\frac{1}{N-1}}_{R_{t}}(\gamma)   \geq  \sigma_{K,N-1}^{(1-t)}(R_{1}-R_{0})   h^{-\frac{1}{N-1}}_{R_{0}}(\gamma)  
 +  \sigma_{K,N-1}^{(t)}(R_{1}-R_{0})h^{-\frac{1}{N-1}}_{R_{1}}(\gamma).  
\end{aligned}
\end{equation}
for $\nu$-a.e. $\gamma \in G$ and all $t \in [0,1]$.

Consider $0< r_{0}<r_{1} \leq R$ and the following probability measure
\[
\mu_{0} : = \frac{1}{r_{1}-r_{0}}\int_{(r_{0}, r_{1})} \frac{m_{r}}{s_{r}} dr.
\]
Let $[0,1] \ni t \mapsto \mu_{t} \in \mathcal{P}_{2}(M,d,m)$ be the geodesic connecting $\mu_{0}$ to $\mu_{1} = \delta_{x_{0}}$ 
with $\mu_{t} = \r_{t}m$. Let moreover $\pi_{t} \in \Pi(\mu_{0},\mu_{t})$ the corresponding optimal coupling.

\begin{proposition}\label{P:cdgeod}
Fix $t \in [0,1)$. Then for $\pi_{t}$-a.e. $(z_{0},z_{1}) \in M^{2}$ the following holds true
\begin{equation}\label{E:vera}
\r_{ts}(\gamma_{s}(z_{0},z_{1}))^{-1/N} \geq \r_{0}(z_{0})^{-1/N} \tau_{K,N}^{(1-s)}(d(z_{0},z_{1})) 
+ \r_{t}(z_{1})^{-1/N} \tau_{K,N}^{(s)}(d(z_{0},z_{1})),
\end{equation}
for every $s \in [0,1]$, where $\gamma_{s}(z_{0},z_{1})$ is the $s$-intermediate point on the geodesic $\gamma$ connecting $z_{0}$ to $z_{1}$.
\end{proposition}

\begin{proof}
We use the following notation: for a given $R$ consider the geodesic $(p_{R,r})_{s\in [0,R]}$ with $p_{R,0} = \delta_{x_{o}}$ and $p_{R,R} = m_{R}$.
The same rule will apply to densities $h_{R,r}$.

Let $[0,1] \ni s \mapsto \Gamma_{st} : = \mu_{st}$ and observe that
\begin{equation}\label{E:geo}
\Gamma_{s}= \mu_{st}= \frac{1}{(1-st)(r_{1}-r_{0})}\int_{(1-st)( r_{0}, r_{1})} h_{r/(1-t), r} m_{r} dr.
\end{equation}

Consider $x_{0} \in M_{\bar r}$ with $r_{0} \leq \bar r \leq r_{1}$. Then the unique $x_{1}$ such that $(x_{0},x_{1})$ is in the support of the optimal plan $\pi_{t}$, belongs to $M_{(1-t)\bar r}$.
Then from Theorem \ref{T:surface} and  \eqref{E:geo} 
\[ 
\begin{aligned}
\big( (r_{1}-r_{0})\r_{st}(\gamma_{s}(x_{0},x_{1})\big)^{-1/N} = &~ \Big( \frac{1}{1-st} h_{\bar r, (1-st)\bar r }(\gamma) \Big)^{-1/N} \crcr
= &~ \Big(\frac{1}{(1-t)s +1-s} \Big)^{-\frac{1}{N}} \Big(h^{-\frac{1}{N-1}}_{\bar r, (1-t)s\bar r +(1-s)\bar r}(\gamma)) \Big)^{\frac{N-1}{N}}  \crcr
\geq &~ (1-s)^{1/N} \Big(\sigma_{K,N-1}^{(1-s)}(t \bar r) h^{-\frac{1}{N-1}}_{\bar r, \bar r  }(\gamma)   \Big)^{\frac{N-1}{N}}  \crcr
+ &~ ((1-t)s)^{1/N} \Big(\sigma_{K,N-1}^{(s)}(t \bar r) h^{-\frac{1}{N-1}}_{\bar r, (1-t)\bar r}(\gamma)   \Big)^{\frac{N-1}{N}} \crcr
=&~ \tau_{K,N}^{(1-s)}(d(z_{0},z_{1}))  \big((r_{1} - r_{0})\r_{0}(x_{0})\big)^{-1/N} \crcr
+ &~  \tau_{K,N}^{(s)}(d(z_{0},z_{1})) \big((r_{1} - r_{0})\r_{t}(z_{1})\big)^{-1/N}
\end{aligned}
\]
The claim follows.
\end{proof}

So far we have proven that given $\mu_{0}: = m(A)^{-1} m\llcorner_{A} $, $x_{0} \in \supp[m]$ and the unique geodesic 
$[0,1] \ni t \mapsto \Gamma(t)$ such that $\Gamma(0) =\mu_{0}$, $\Gamma(1) =\delta_{x_{0}}$ and $\Gamma(t) = \r_{t}m$ for $t \in [0,1)$
we have for any $t \in [0,1)$:
\[
\begin{aligned}
\mathcal{S}_{N'}(\Gamma(ts)|m) \leq   - \int_{M\times M}& \Big[ \tau_{K,N'}^{(1-s)}(d(x_{0},x_{1}))\r_{0}^{-1/N'}(x_{0})  \crcr
&~ + \tau_{K,N'}^{(s)}(d(x_{0},x_{1}))\r_{t}^{-1/N'}(x_{1})  \Big]  \pi_{t}(dx_{0}dx_{1}),
\end{aligned}
\]
for all $s \in [0,1]$ and all $N'\geq N$, where $\pi_{t}  =  (P_{0},P_{t})_{\sharp} \Xi$.

We are ready to prove the main theorem of this chapter.
\begin{theorem}\label{T:main}
Let $(M,d,m)$ be a non-branching metric measure spaces satisfying $\mathsf{CD}_{loc}(K,N)$. Then 
$(M,d,m)$ satisfies $\mathsf{MCP}(K,N)$.
\end{theorem}

\begin{proof}
{\it Step 1.}
Let $\gamma : M^{2} \to \G(M)$ be the map introduced in Lemma \ref{L:map} and define for each $t \in [0,1]$
a Markov kernel $Q_{t}$ from $M^{2}$ to $M$ by
\[
Q_{t}(x,y;B) : = 1_{B}(\gamma_{t}(x,y))
\]
and for each pair $t,x$ a measure $m_{t,x} = \int Q_{t}(x,y; \cdot) m(dy)$.

For each $x \in M$ let $M_{x}$ denote the set of all $y \in M$ for which there exists a unique geodesic connecting $x$ and $y$ 
and let $M_{0}$ be the set of $x$ such that $m(M\setminus M_{x})=0$. By assumption $m(M\setminus M_{0})=0$.

{\it Step 2.}
Fix $x_{0} \in M_{0}$ and $B \subset M$. Put $A_{0}:= \gamma_{t}(x_{0},\cdot)^{-1}(B)$ and $\mu_{0}: =m(A_{0})^{-1} m\llcorner_{A_{0}}$.
Considering $s=1$ in \eqref{E:vera} it follows that 
\[
m(B)^{1/N} \geq \inf_{y \in A_{0}} \tau_{K,N}^{(t)}(d(y,x_{0})) m(A_{0})^{1/N},
\]
or equivalently
\[
m(B) \geq \inf_{y \in \gamma_{t}(x_{0},\cdot)^{-1}(B)} \varsigma_{K,N}^{(t)}(d(y,x_{0})) m(\gamma_{t}(x_{0},\cdot)^{-1}(B)) = 
\inf_{z\in B} \varsigma^{(t)}_{K,N}\bigg(\frac{d(z,x_{0})}{t}\bigg) m_{t,x_{0}}(B).
\]
Decomposing $B$ into a disjoint union $\cup_{i}B_{i}$ with $B_{i} = B \cap (\bar B_{\ve i}(x_{0}) \setminus \bar B_{\ve (i-1)}(x_{0})$, 
and applying the previous estimate to each  of the $B_{i}$ we obtain as $\ve \to 0$
\[
m(B) \geq \int_{B} \varsigma^{(t)}_{K,N} \bigg( \frac{d(z,x_{0})}{t}\bigg) m_{t,x_{0}}(dz)
\]
or equivalently
\[
m(B) \geq \int_{B} \varsigma^{(t)}_{K,N} ( d(z,x_{0} )) Q_{t}(x_{0},y; B)m(dy).
\]
\end{proof}

\section*{Outlook}

In the last part of this note we sketch the most general case we can address using the approach introduced so far. 
We start recalling the definition of $d$-transform:  
for $f : M \to \bar \erre$ Borel measurable 
\[
f^{d}(y) : = \inf_{x \in M}\frac{d^{2}(x,y)}{2} - f(x),
\]
$f^{d}$ is the $d$-transform of $f$. Accordingly, a map is $d$-concave if it can be written as the $d$-transform of another map.

\subsection*{The Setting}
Let $A \subset M$ be a Borel set and define  the map $\f_{A}(x) : =d^{2}(A,x)/2$
where $d(A,x) : = \inf \{ d(z,x) : z \in A\}$. Clearly $\f_{A}$ is $d$-concave, indeed if 
\[
\infty_{A} (x) : = 
\begin{cases} 0 & x \in A  \crcr 
+ \infty & x \notin  A,
\end{cases}
\]
then $\f_{A} = \infty_{A}^{d}$.

\begin{definition}\label{D:convex}
Let $A \subset M$ be a closed set. The set $A$ is \emph{d-convex} if 
\begin{equation}\label{E:conv}
(-\infty_{A})^{dd} = - \infty_{A}.
\end{equation}
\end{definition}

\begin{remark}
In the Euclidean case, i.e. $\erre^{n}$ equipped with euclidean distance, Definition \ref{D:convex} is equivalent to the standard notion of convexity.
Indeed for $f : \erre^{n} \to \bar \erre$ with $f > -\infty$ and not identically $+\infty$, consider the Legendre transform
\[
f^{*}(y) : = \sup_{x \in \erre^{n}} \langle y,x\rangle - f(x).
\]
Then it is well-know that $f^{**} = f$ if and only if $f$ is convex and l.s.c. (see for instance \cite{rachevru:transport}, Chapter 3). Since
\[
(-f)^{dd}(x) = \| x \|^{2} - (f + \|\cdot \|^{2})^{**}(x), 
\]
it is fairly easy to conclude that $(-\infty_{A})^{dd} = - \infty_{A}$ is equivalent to convexity, provided $A$ is a closed set.
\end{remark}

We will prove the analogous of Proposition \ref{P:cdgeod} only for those optimal transport plan having $(\f_{A}, -\infty_{A})$ as Kantorovich potentials.
Define the set  
\[
\Gamma_{A} : = \{ (x,y) \in M \times M : \f_{A}(x) -\infty_{A}(y) = d^{2}(x,y) \} = \left\{ (x,y) \in M\times A : \f_{A}(x) = \frac{d^{2}(y,x)}{2}\right\}
\]
and the corresponding family of optimal dynamical transference plan 
\[
\gammA_{A}: = \{ \gammA \in \mathcal{P}(\G(M)) : (e_{0},e_{1})_{\sharp}(\gammA)(\Gamma_{A}) = 1, e_{0\,\sharp}\gammA= \r_{0}m \}. 
\]

\begin{theorem}
Let $A \subset M$ be compact and $d$-convex. Then every $\gammA \in \gammA_{A}$ satisfies the following: 
for every $t \in [0,1)$ we have $e_{t\,\sharp} \gammA = \r_{t}m$ and
\[
\begin{aligned}
\mathcal{S}_{N'}(e_{ts\,\sharp}\gammA|m) \leq   - \int_{M\times M}& \Big[ \tau_{K,N'}^{(1-s)}(d(x_{0},x_{1}))\r_{0}^{-1/N'}(x_{0})  \crcr
&~ + \tau_{K,N'}^{(s)}(d(x_{0},x_{1}))\r_{t}^{-1/N'}(x_{1})  \Big]  \pi_{t}(dx_{0}dx_{1}),
\end{aligned}
\]
for all $s \in [0,1]$ and all $N'\geq N$, where $\pi_{t}  =  (e_{0},e_{t})_{\sharp} \gammA$.
\end{theorem}

We present an outline of the proof.

\begin{proof}
Due to non-branching assumption, 
$\mathsf{CD}(K,N)$ implies $\mathsf{CD}_{LV}(K,N)$, introduced by Lott and Villani in \cite{villott:curv}.  
The latter implies that every geodesic consists of absolute continuous measures at intermediate times, 
whenever one of the two endpoints is absolute continuous (see \cite{villa:Oldnew} Theorem 30.19). 

The proof of this result is preserved if we replace all the coefficients $\tau_{K,N}$ by coefficients $\sigma_{K,N}$.
The corresponding curvature-dimension condition $\mathsf{CD}^{*}_{LV}(K,N)$ follows from our condition $\mathsf{CD}_{LV}(K,N)$,  
due to the non-branching assumption. It follows therefore that 
\[
e_{t\,\sharp} \gammA = \r_{t} m, \qquad \forall t \in [0,1).
\]
Note that  $\mu_{1} (\partial A) = 1$. Therefore every geodesic belonging to the support of $\gammA$ never enters inside $A$.

\subsection*{Polar Coordinates} Consider the following set 
\[
\Gamma_{A}(1): = \{ \gamma_{t} : (\gamma,t) \in \supp(\gammA) \times [0,1) \}. 
\]
We only need a disintegration of $m$ restricted to $\Gamma_{A}(1)$. Denote with
\[
B_{r}(A) : = \{ x : d_{A}(x)\leq r\}.
\]
Consider the family $\{\partial B_{r}(A) \}_{r > 0}$ giving a partition of $\Gamma_{A}(1)$. It follows that 
\[
m\llcorner_{\Gamma_{A}(1)} = \int_{(0,\infty)}  \bar m_{r} q(dr), \qquad \bar m_{r}(\partial B_{r}(A))=1.
\]

Since the map $d(A,x)$ is Lipschitz and with strictly positive upper gradient on $\partial B_{r}(A)$ for $r>0$, 
it follows from the coarea formula in metric measure spaces (see Proposition 4.2 of \cite{miranda:bvcoarea}) that  
\begin{equation}\label{E:disintA}
m\llcorner_{\Gamma_{A}(1)} = \int_{(0,\infty)}  \bar m_{r} q(r) \mathcal{L}^{1}(dr) = 
\int_{(0,\infty)}   m_{r} \mathcal{L}^{1}(dr).
\end{equation}

\subsection*{Estimate in codimension 1}
In the same way we disintegrate $\gammA$:
\[
\gammA = \int \gammA_{r} dr, \qquad \|\gammA_{r}\|^{-1} \gammA_{r}(\{ \gamma : d(A,\gamma_{0})=r \})=1.
\]  
So fix $R$ and consider the constant speed geodesic $(p_{r})_{r\in [0,R]}$ such that $p_{R} = e_{0\,\sharp}\|\gammA_{R}\|^{-1} \gammA_{R}$
and $p_{0} = e_{1\,\sharp}\|\gammA_{R}\|^{-1} \gammA_{R}$.
Since $\mu_{t} \ll m$ for every $t \in [0,1)$ it follows that 
\[
p_{r} =  h_{r} m_{r}.
\]
Now we can consider the family of geodesics \eqref{E:curve} depending on $R_{i}$ and $L_{i}$ for $i=1,2$.
The very same proof of Theorem \ref{T:surface} gives 
for sufficiently close $0<R_{0}<R_{1}$ the following:
\[
h_{R_{1/2}}^{-\frac{1}{N-1}}(\gamma)  
 \geq  \sigma_{K,N-1}^{(1/2)}(R_{1}-R_{0}) \left\{   h_{R_{0}}^{-\frac{1}{N-1}}(\gamma) +   h_{R_{1}}^{-\frac{1}{N-1}}(\gamma) \right\}.
\]
As already shown during this note, the above estimates passes from local to global and therefore it holds true for any $0< R_{0} <R_{1}$.

\subsection*{Full dimensional estimate} We use the following notation: 
for a given $R$ consider the geodesic $(p_{R,r})_{s\in [0,R]}$ with $p_{R,0} = e_{1\,\sharp} (\|\gammA_{R}\|^{-1}\gammA_{R})$ and 
$p_{R,R} = e_{0\,\sharp} (\|\gammA_{R}\|^{-1}\gammA_{R})$.
Then for every $t\in [0,1)$:
\[
p_{R,(1-t)R} = \frac{\r_{t}}{\int \r_{t}m_{(1-t)R}} m_{(1-t)R} .
\]
Since 
\[
\int_{\{R>0\}} m_{(1-t)R} \mathcal{L}^{1}(dR) = \frac{1}{1-t} m\llcorner_{\Gamma_{A}(1)},
\]
we find the information on the density of the missing direction in the following way: 
\[
\mu_{t} =  \int_{\{R>0\}} \frac{\r_{t}}{\int \r_{t}m_{(1-t)R}}  m_{(1-t)R}   \bigg( \int \r_{t}m_{(1-t)R} \bigg) dR = 
\frac{1}{(1-t)} \int_{\tau>0} \frac{\r_{t}}{\int \r_{t}m_{\tau}}\bigg( \int \r_{t}m_{\tau} \bigg) m_{\tau}  d\tau.
\]
It follows that the inverse of the 1-dimensional density evolves linearly with $t$. Imitating the proof of Proposition \ref{P:cdgeod},
for $t \in [0,1)$, for $(e_{0},e_{t})_{\sharp}\gammA$-a.e. $(z_{0},z_{1}) \in M^{2}$ the following holds true
\[
\r_{ts}(\gamma_{s}(z_{0},z_{1}))^{-1/N} \geq \r_{0}(z_{0})^{-1/N} \tau_{K,N}^{(1-s)}(d(z_{0},z_{1})) 
+ \r_{t}(z_{1})^{-1/N} \tau_{K,N}^{(s)}(d(z_{0},z_{1})),
\]
for every $s \in [0,1]$, where $\gamma_{s}(z_{0},z_{1})$ is the $s$-intermediate point on the geodesic $\gamma$ connecting $z_{0}$ to $z_{1}$.
Integrating the previous inequality, we have the claim.
\end{proof}

To conclude this note we want to list the differences between the general globalization theorem, our case and the measure contraction-property.
Assume that $(M,d,m)$ satisfies $\mathsf{CD}_{loc}(K,N)$ and let $\mu_{0}$ be an absolutely 
continuos probability measure and $\f$ be a $d$-concave Kantorovich potential for a dynamical optimal transference plan:

\begin{itemize}
\item[$\mathsf{MCP}$:]  prove \eqref{E:vera} for every $\mu_{0}$ and every $\f = (-\infty_{\{z\}})^{d}$, for every $z \in M$;   
\item [\emph{Note}:]   prove \eqref{E:vera} for every $\mu_{0}$ and every $\f = (-\infty_{A})^{d}$, with $A$ $d$-convex;  
\item [$\mathsf{CD}$:]   prove \eqref{E:vera} for every $\mu_{0}$ and every $\f$;
\end{itemize}
where by $Note$ we mean the level of generality obtained in this paper. It is clear that $\mu_{1}$ is determined by the choice of $\mu_{0}$
and $\f$.

\end{document}